\documentclass[11pt]{amsart}
\usepackage{
amssymb
}
\newcommand{\no}[1]{#1}

\renewcommand{\no}[1]{}

\no{\usepackage{times}\usepackage[
subscriptcorrection, slantedGreek, 
nofontinfo]{mtpro}
\renewcommand{\Delta}{\upDelta}
}

\newtheorem{theorem}{Theorem}
\newtheorem{proposition}{Proposition}
\newtheorem{lemma}{Lemma}

\newtheorem{corollary}{Corollary}

\theoremstyle{remark}
\newtheorem{remark}{Remark}

\DeclareMathOperator{\Vol}{Vol}

\DeclareMathOperator{\supp}{supp}

\DeclareMathOperator{\WF}{WF}
\DeclareMathOperator{\dist}{dist}

\newcommand{\eps}{\varepsilon}

\newcommand{\R}{{\bf R}}
\newcommand{\Id}{\mbox{Id}}
\renewcommand{\r}[1]{(\ref{#1})}
\newcommand{\PDO}{$\Psi$DO}
\newcommand{\be}[1]{\begin{equation}\label{#1}}
\newcommand{\ee}{\end{equation}}

\renewcommand{\d}{\mathrm{d}}

\renewcommand{\i}{\mathrm{i}}

\newcommand{\bo}{\partial \Omega}

\title[Thermoacoustic tomography]{Thermoacoustic tomography with variable sound speed}

\author[P. Stefanov]{Plamen Stefanov}
\address{Department of Mathematics, Purdue University, West Lafayette, IN 47907}
\thanks{First author partly supported by a NSF  Grant DMS-0800428}

\author[G. Uhlmann]{Gunther Uhlmann}
\address{Department of Mathematics, University of Washington, Seattle, WA 98195}
\thanks{Second author partly supported by a NSF FRG grant No.~0554571 and a Walker Family Endowed Professorship}

\usepackage{graphicx}
\graphicspath{%
    {converted_graphics/}
    {./}
}
\begin{document}
\maketitle

\begin{abstract}
We study the mathematical model of thermoacoustic tomography in media with a variable speed for a fixed time interval $[0,T]$ so that all signals issued from the domain leave it after time $T$. In case of measurements on the whole boundary, we give an explicit solution in terms of a Neumann series expansion. We give almost necessary and sufficient conditions for uniqueness and stability when the measurements are taken on a part of the boundary. 
\end{abstract}

\maketitle

\section{Introduction}  
In thermoacoustic tomography, a short electro-magnetic  pulse is sent through a patient's body. The tissue reacts and emits an ultrasound wave from any point, that is measured away from the body. Then one tries to reconstruct the internal structure of a patient's body form those measurements, see e.g, \cite{Haltmeier04, Haltmeier05, Kruger03, Kruger99,XuWang06}. For more detail, an extensive list of references, and the recent progress in the mathematical understanding of this problem, we refer to \cite{AgrKuchKun2008,FinchRakesh08,Hristova08, HristovaKu08,KuchmentKun08, Patch04}. Both constant and non-constant sound speeds have been studied and naturally, the results are more complete in the constant speed case. 

The purpose of this work is to study this problem under the assumption of a variable speed. We will actually  formulate the problem in anisotropic media. 
Let $g$ be a  Riemannian metric in $\R^n$,  let $a$ be a  vector field, and let  $c>0$, $q\ge0$ be  functions, all smooth and real valued. Assume for convenience that $g$ is Euclidean outside a large compact, and $c-1=q=a=0$ there (since we work with $t$ in a fixed interval, by the finite speed of propagation, this assumption is not essential). 
Let $P$ be the differential operator
\be{P}
P = c^2\frac{1}{\sqrt{\det g}}\left(\frac1\i \frac{\partial}{\partial x^i }+ a_i \right)g^{ij}\sqrt{\det g} 
\left(\frac1\i \frac{\partial}{\partial x^j }+ a_j \right)
+q .
\ee
Let $u$ solve the problem
\begin{equation}   \label{1}
\left\{
\begin{array}{rcll}
(\partial_t^2 +P)u &=&0 &  \mbox{in $(0,T)\times \R^n$},\\
u|_{t=0} &=& f,\\ \quad \partial_t u|_{t=0}& =&0, 
\end{array}
\right.               
\end{equation}
where $T>0$ is fixed. 

Assume that $f$ is supported in $\bar\Omega$, where $\Omega\subset \R^n$ is some smooth bounded domain. The measurements are modeled by the operator
\be{1b}
\Lambda f : = u|_{[0,T]\times\partial\Omega}.
\ee
The problem is to reconstruct the unknown $f$. 

The presence of the magnetic field $\{a_j\}$ is perhaps of no interest for applications but it does not cause any additional  difficulties. 

If $T=\infty$, then one can solve a problem with Cauchy data $0$ at $t=\infty$ (as a limit), and boundary data $h=\Lambda f$. The zero Cauchy data are justified by  local energy decay that  holds for non-trapping geometry, for example (actually, it is always true but much weaker and not uniform in general). Then solving the resulting problem backwards recovers $f$. Now, based on that, one can show that for a fixed $T$, one can still do the same thing with an error $\epsilon(T)\to0$,  as $T\to\infty$. This is known as the  time reversal method. In the non-trapping case, $n$ odd, the error is uniform and $\epsilon(T)=O(e^{-T/C})$. There is no good control over $C$ though. Error estimates based on local energy decay can be found in \cite{Hristova08}, see also Corollary~\ref{cor_est}. Other reconstruction methods have been used as well, see, e.g., \cite{HristovaKu08} for a discussion, and they all use measurements for all $t$ in the variable coefficients case, i.e., $T=\infty$; and they are only approximate for $T<\infty$ with an error depending on the local energy decay rate. Of course, if $n$ is odd  and $P=-\Delta$, any finite $T>\mbox{diam}(\Omega)$ suffices by the Huygens'  principle. 

We refer to Section~\ref{sec_incomplete} for a discussion of uniqueness results. 

In this paper, we want to study what happens when $T<\infty$ is fixed,  greater than the length of the longest geodesic in $\Omega$ (thus the metric $c^{-2}g$ is assumed to be non-trapping). In case of measurements on the whole boundary, our main result is that the problem is Fredholm, uniquely solvable, and can be solved explicitly with a Neumann series expansion. In case of partial data, in Section~\ref{sec_incomplete} we give an almost necessary and sufficient condition for uniqueness, and another almost  necessary and sufficient condition for stability. In Proposition~\ref{pr_FIO} we characterize  $\Lambda$ as a sum of two Fourier Integral Operators with canonical relations of graph type. 

\section{Complete data}
Notice first that $P$ is formally self-adjoint w.r.t.\ the measure $c^{-2}\d\Vol$, where $\d\Vol(x) = \sqrt{\det g}\, \d x$.  
Given a domain $U$, and a function $u(t,x)$, define the energy
\[
E_U(t,u) = \int_U\left( |Du|^2  +c^{-2}q|u|^2 +c^{-2}|u_t|^2 \right)\d\Vol,
\]
where $D_j= -\i\partial/\partial x^j+a_j$, $D=(D_1,\dots,D_n)$, $|Du|^2=g^{ij}(D_iu)(D_ju)$, and $\d\Vol(x) = (\det g)^{1/2}\d x$. 
In particular, we define the space $H_{D}(U)$ to be the completion of $C_0^\infty(U)$ under the Dirichlet norm
\be{2.0H}
\|f\|_{H_{D}}^2= \int_U \left(|Du|^2 +c^{-2}q|u|^2 \right)\,\d\Vol.
\ee
It is easy to see that $H_{D}(U)\subset H^1(U)$, if $U$ is bounded with smooth boundary, therefore, $H_{D}(U)$ is topologically equivalent to $H_0^1(U)$. If $U=\R^n$, this is true for $n\ge3$ only, if $q=0$.  By the finite speed of propagation, the solution with compactly supported Cauchy data always stays in $H^1$ even when $n=2$.  
The energy norm for the Cauchy data $(f,\psi)$, that we denote by $\|\cdot\|_{\mathcal{H}}$ is then defined by
\[
\|(f,\psi)\|^2_{\mathcal{H}} = \int_U\left( |Df|^2 +c^{-2}q|f|^2  +c^{-2}|\psi|^2 \right)\d\Vol.
\]
This defines the energy space 
\[
\mathcal{H}(U) = H_D(U)\oplus L^2(U).
\] 
Here and below, $L^2(U) = L^2(U; \; c^{-2}\d\Vol)$. Note also that 
\be{Pf}
\|f\|^2_{H_D} = (Pf,f)_{L^2}.
\ee
The wave equation then can be written down as the system
\be{s1}
\mathbf{u}_t= \mathbf{P}\mathbf{u}, \quad \mathbf{P} = \begin{pmatrix} 0&I\\P&0 \end{pmatrix},
\ee
where $\mathbf{u}=(u,u_t)$ belongs to the energy space $\mathcal{H}$. The operator $\mathbf{P}$ then extends naturally to a skew-selfadjoint operator on $\mathcal{H}$. In this paper, we will deal with either $U=\R^n$ or $U=\Omega$. In the latter case, the definition of $H_D(U)$ reflects Dirichlet boundary conditions. 

One method to get an approximate solution of the thermoacoustic problem is the following time reversal method, that is actually used in a modified form, see the comments below. Given $h$, let $v_0$ solve
\begin{equation}   \label{2.1}
\left\{
\begin{array}{rcll}
(\partial_t^2 + P)v_0 &=&0 &  \mbox{in $(0,T)\times \Omega$},\\
v_0|_{[0,T]\times\partial\Omega}&= &h,\\
v_0|_{t=T} &=& 0,\\ \quad   \partial_t v_0|_{t=T}& =&0. \\
\end{array}
\right.               
\end{equation}
Then we define the following ``approximate inverse''
\[
A_0 h := v_0(0,\cdot) \quad \mbox{in $\bar\Omega$}.
\]
Then $A_0\Lambda f$ is viewed as an approximation to $f$. As we mentioned above, that is actually true asymptotically as $T\to\infty$, with the modified version of the time reversal method described below, (see \cite{Hristova08}) but $T$ is fixed in our analysis. 

In this form, the time reversal method has the following downside: $h$ may not vanish on $\{T\}\times \bo$, therefore the mixed problem above has boundary data with a possible  jump type of singularity at $\{T\}\times\bo$ (the compatibility conditions might be violated). That singularity will propagate back to $t=0$ and will  affect $v_0$, and then $v_0$ may not be in the energy space. The operator $A_0\Lambda$ may fail to be Fredholm or even bounded then, and in particular $A_0\Lambda f$ might be more singular than $f$. For this reason, $h$ is usually  cut off smoothly near $t=T$, i.e., $h$ is replaced by $\chi(t) h(t,x)$, where $\chi\in C^\infty(\R)$, $\chi=0$ for $t=T$, and $\chi=1$ in a neighborhood of  $(-\infty, T(\Omega))$, see e.g., \cite[Section~2.2]{Hristova08}.

We will modify this approach in a way that would make the problem Fredholm, and will make the error operator a contraction. 
To this end, we proceed as follows. Given $h$ (that eventually will be replaced by $\Lambda f$), solve
\begin{equation}   \label{2}
\left\{
\begin{array}{rcll}
(\partial_t^2 + P)v &=&0 &  \mbox{in $(0,T)\times \Omega$},\\
v|_{[0,T]\times\partial\Omega}&= &h,\\
v|_{t=T} &=& \phi,\\ \quad \partial_t v|_{t=T}& =&0, \\
\end{array}
\right.               
\end{equation}
where $\phi$ solves the elliptic boundary value problem
\be{3}
P\phi=0, \quad 
\phi|_{\partial\Omega} = h(T,\cdot).
\ee 
Since $P$ is a positive operator, $0$ is not a Dirichlet eigenvalue of $P$ in $\Omega$, and therefore \r{3} is uniquely solvable. Note that the initial data at $t=T$ satisfy compatibility conditions of first order (no jump at $\{T\}\times\bo$). 
Then we define the following pseudo-inverse
\be{4}
A h := v(0,\cdot) \quad \mbox{in $\bar\Omega$}.
\ee
The operator $A$ maps continuously the closed subspace of $H^1([0,T]
\times \bo)$ consisting of functions that vanish at $t=T$ (compatibility condition) to $H^1(\Omega)$, see \cite{LasieckaLT}. 
It also sends the range of $\Lambda$ to $H^1_0(\Omega) \cong H_D(\Omega)$, as the proof below indicates.

In the next theorem and everywhere below, $T(\Omega)$ is the supremum of the lengths of all geodesics of the metric $c^{-2}g$ in $\bar\Omega$. Also, $\dist(x,y)$ denotes the distance function in that metric. We then call $(\Omega, c^{-2}g)$ non-trapping, if $T(\Omega)<\infty$. 

\begin{theorem}  \label{thm2.1} Let $(\Omega, c^{-2}g)$ be non-trapping, and let $T>T(\Omega)$. Then 
$A\Lambda=\Id-K$, where $K$ is compact in $H_{D}(\Omega)$, and   $\|K\|_{H_{D}(\Omega)}<1$. 
In particular, $\Id-K$ is invertible on $H_{D}(\Omega)$, and the inverse thermoacoustic problem has an explicit solution of the form
\be{2.2}
f = \sum_{m=0}^\infty K^m A h, \quad h:= \Lambda f.
\ee
\end{theorem}
\begin{proof}
Let $w$ solve
\begin{equation}   \label{2.3}
\left\{
\begin{array}{rcll}
(\partial_t^2 + P)w &=&0 &  \mbox{in $(0,T)\times \Omega$},\\
w|_{[0,T]\times\partial\Omega}&= &0,\\
w|_{t=T} &=& u|_{t=T}-\phi,\\ \quad w_t|_{t=T}& =&u_t|_{t=T},\\
\end{array}
\right.               
\end{equation}
where $u$ solves \r{1} with a given $f\in H_D$. 
Let $v$ be the solution of \r{2} with $h=\Lambda f$. Then $v+w$ solves the same initial boundary value problem in $[0,T]\times\Omega$ that $u$ does (with initial conditions at $t=T$), therefore $u=v+w$. Restrict this to $t=0$ to get
\[
f= A\Lambda f + w(0,\cdot).
\]
Therefore, 
\[
Kf = w(0,\cdot).
\]
In what follows, $(\cdot,\cdot)_{H_{D}(\Omega)}$ is the inner product in $H_{D}(\Omega)$, see \r{2.0H}, applied to functions that belong to $H^1(\Omega)$ but maybe not to $H_{D}(\Omega)$ (because they may not vanish on $\bo$). Set $u^T := u(T,\cdot)$. By \r{Pf} and the fact that $u^T=\phi$ on $\bo$, we get
\[
(u^T-\phi,\phi)_{H_{D}(\Omega)}=0.
\]
Then
\[
\|u^T-\phi\|^2_{H_{D}(\Omega)} = \|u^T\|^2_{H_{D}(\Omega)} - \|\phi\|^2_{H_{D}(\Omega)}\le \|u^T\|^2_{H_{D}(\Omega)}.
\]
Therefore, the energy  of the initial conditions in \r{2.3} satisfies the inequality
\be{2.4}
E_\Omega(w,T) = \|u^T-\phi\|^2_{H_{D}(\Omega)}  +\|u^T_t\|^2_{L^2(\Omega)}  \le E_\Omega(u,T).
\ee
Since the Dirichlet boundary condition is energy preserving, we get that 
\[
E_{\Omega}(w,0) =    E_{\Omega}(w,T)\le  E_{\Omega}(u,T)\le E_{\R^n}(u,T)= E_{\Omega}(u,0) = \|f\|^2_{H_{D}(\Omega)}. 
\]
In particular, 
\be{2.5}
\|Kf\|^2_{H_{D}(\Omega)} \le E_{\Omega}(w,0)\le \|f\|^2_{H_{D}(\Omega)}.
\ee

We show next that actually the inequality above is strict, i.e., 
\be{2.6}
\|Kf\|_{H_{D}(\Omega)} < \|f\|_{H_{D}(\Omega)}, \quad f\not=0.
\ee
Assume the opposite. Then for some $f\not=0$, all inequalities leading to \r{2.5} are equalities. In particular, $ E_{\Omega}(w,T)=  E_{\R^n}(u,T)$. Then 
\[
u(T,x) = 0, \quad \mbox{for $x\not\in\Omega$}.
\]
By the finite domain of dependence then
\be{2.7}
u(t,x) = 0 \quad \mbox{when $\dist(x,\Omega)>|T-t|$}.
\ee
One the other hand, we also have 
\be{2.8}
u(t,x) = 0 \quad \mbox{when $\dist(x,\Omega)>|t|$}.
\ee
Therefore,
\be{2.8a}
u(t,x) = 0 \quad \mbox{when $\dist(x,\bo)>T/2, \;  -T/2\le t\le 3T/2$}.
\ee
Since $u$ extends to an even function of $t$ that is still a solution of the wave equation, we get that \r{2.8a} actually holds for $|t|<3T/2$. 
Then one concludes by Tataru's theorem, see Theorem~\ref{thm_uq}, that $u=0$ on $[0,T]\times \Omega$, therefore, $f=0$.  We refer to \cite{FinchRakesh08} for a similar argument. Note that the time interval here is actually larger than what we need for the uniqueness argument, see also Theorem~\ref{thm_uniq} and Corollary~\ref{cor_1} below. 

We will show now that $K$ is compact. Since $T>T(\Omega)$, all singularities starting from $\bar \Omega$ leave $\bar \Omega$ at $t=T$. Therefore, $u(T,\cdot)$ and $u_t(T,\cdot)$, restricted to $\bar\Omega$, are $C^\infty$. Moreover, considered as linear operators of $f$, they are operators (FIOs, actually) with smooth Schwartz kernels. 
Then so is $\phi$, see \r{3}, by elliptic regularity. 
Therefore, the map $H_{D}(\Omega)\ni f\mapsto u(T,\cdot)-\phi\in H_D(\Omega)$ is compact because it is an operator with smooth kernel on $\bar\Omega$. Next, the map $H_{D}(\Omega)\ni f\mapsto u_t(T,\cdot)\in H_D(\Omega)$ is compact as well. Since the solution operator of \r{2.3} from $t=T$ to $t=0$ is unitary in $H_D(\Omega)\oplus L^2(\Omega)$, we get that the map $H_{D}(\Omega)\ni f\mapsto w(0,\cdot)\in H_D(\Omega)$ is compact, too, as a composition of a compact and a bounded one.

Now, one has 
\be{2.9}
\|Kf\|_{H_{D}(\Omega)} \le \sqrt{\lambda_1} \|f\|_{H_{D}(\Omega)}, \quad f\not=0,
\ee
where $\lambda_1$ is the largest eigenvalue of $K^*K$. Then $\lambda_1<1$ by \r{2.6}.
\end{proof}

\begin{remark}
Although we proved that $K$ is compact, we did not show that $K$ is smoothing of   $1$ degree. Actually, we showed that $K$ is a composition of a smoothing operator and a bounded one. To make $K$ smoothing, we need to modify the initial condition for $w_t(T,\cdot)$ in \r{2.3}, as we did it for $w(T,\cdot)$, so that it would satisfy the compatibility at $\{T\}\times \bo$ (no jump there, i.e, $w_t(T,\cdot)\in H_0^1(\Omega)$). That will put $(w(T,\cdot), w_t(T,\cdot))$ in the domain of the generator of the solution group, in other words,  $(w_t, Pw(T,\cdot))$ would be in the energy space. Then the same would be true for $Pw(0,\cdot)=-PKf$, hence $Kf\in H^2(\Omega)$. Then we get a Fredholm problem again but the norm of $K$ may not be less than $1$ (that still might be true in a suitable norm). In any case, $\Id-K$ will be invertible. 
One can also modify the initial data at $t=T$ in \r{2.3} to satisfy even higher order compatibility condition, and that will increase the smoothing properties of $K$. 
\end{remark}

\begin{remark}
The smoothness requirements on the coefficients of $P$ can be relaxed to require
smoothness of a finite degree. All we need, besides a well posed problem in the energy space, is a propagation of singularities result with a gain of smoothness on $t=T$ enough to guarantee compactness of $K$; and Tataru's uniqueness theorem in that case. We will not pursue this for the sake of simplicity of the exposition. 
\end{remark}

The proof of Theorem~\ref{thm2.1} provides an estimate of the error in the reconstruction if  we  use the first term in \r{2.2} only that is $Ah$. It is in the spirit of \cite{Hristova08} and relates the error to the local energy decay, as can be expected. 
\begin{corollary}\label{cor_est}
\[
\|f-A\Lambda f\|_{H_D(\Omega)} \le \left( \frac{E_\Omega(u,T)}{E_\Omega(u,0)}\right)^\frac12 \|f\|_{H_D(\Omega)} ,\quad \forall f\in H_{D(\Omega)},\; f\not=0,
\]
where $u$ is the solution of \r{1}.
\end{corollary}

Note that the $f-A\Lambda f=Kf$, and the corollary actually provides an upper bound for $\|Kf\|$. The estimate above also can be used to estimate the rate of convergence of the Neumann series \r{2.2} when we have a good control over the uniform local energy decay from time $t=0$ to time $t=T$. The estimate holds even without the non-trapping condition and for any $T>0$ but $E_\Omega(u,T)/E_\Omega(u,0)$, that is always less or equal to $1$,  can be guaranteed to have a uniform upper bound less than $1$ for all $f$ only when $T>T(\Omega)$; then the operator norm of $K$ is less than $1$, as well. If $T(\Omega)/2<  T\le T(\Omega)$, we can only say that $\|Kf\|<\|f\|$ for any $f$, see Corollary~\ref{cor_1}, below but that does not necessarily imply that $\|K\|<1$. If $T<T(\Omega)/2$, then 
there is always $f$ so that that quantity equals $1$ by a trivial domain of dependence argument. 


\section{Incomplete data} \label{sec_incomplete}
The case of partial measurements has been discussed in the literature as well, see e.g.,\cite{KuchmentKun08,XuKA, XuWKA}.  One of the motivations is that in breast  imaging, for example, measurements are possible only on part of the boundary. For simplicity, we assume in this section that $P=-\Delta$  outside $\Omega$; in particular $c=1$ and $g$ is Euclidean outside $\Omega$:  
\be{i0}
c(x)=1, \quad  g_{ij}(x)=\delta_{ij}, \quad\mbox{for $x\not\in\Omega$}.
\ee
All geodesics below are related to the metric $c^{-2}g$.

Let $\Gamma\subset\bo$ be a relatively open subset of $\bo$. Set 
\be{i1}
\mathcal{G} := \left\{  (t,x); \; x\in \Gamma, \, 0<t<s(x)    \right\},
\ee
where $s$ is a fixed continuous function on $\Gamma$. This corresponds to measurements taken at each $x\in\Gamma$ for the time interval $0<t<s(x)$. The special case studied so far is  $s(x)\equiv T$, for some $T>0$; then $\mathcal{G}= [0,T]\times\Gamma$. 

We assume now that the observations are made on $\mathcal{G}$ only, i.e., we assume we are given
\be{3.1}
\Lambda f|_{\mathcal{G}},
\ee
where, with some abuse of notation, we denote by $\Lambda$ the operator in \r{1b}, with $T=\infty$ (that actually can be replaced by any upper bound of the function $s$). 
 Then we  want to know under what conditions one can recover $f$, and when  that recovery is stable. 

Uniqueness and reconstruction results in the constant coefficients case  based on spherical means were known for a while, see e.g., the review paper \cite{KuchmentKun08}. 
If $P=-c^2(x)\Delta$, and $\mathcal{G} = [0,T]\times \bo$, Finch and Rakesh \cite{FinchRakesh08} 
have proved that $\Lambda f$ recovers $f$ uniquely as long as $T>T(\Omega)$. A uniqueness result when $\Gamma$ is a part of $\bo$ in the constant coefficients case is given in \cite{finchPR}, and we follow the ideas of that proof below.  The Holmgren's  uniqueness theorem for constant coefficients  and its analogue for variable ones, see Theorem~\ref{thm_uq} below,  play a central role in the proofs that suggests possible instability without further assumptions, see also the remark following Theorem~\ref{thm_stab} below. Stability of the reconstruction when $P=-\Delta$ and $T=\infty$  follows from the known reconstruction formulas, see e.g., \cite{KuchmentKun08}. In the variable coefficients case, stability estimates as $T\to\infty$ based on local energy decay  have been established recently in \cite{Hristova08}. When $T$ is fixed, there is the general feeling that if one can recover ``stably'' all singularities, and if there is uniqueness, there must be stability (although this has been viewed from the point of view of integral geometry, see also Section~\ref{sec_ig}).   We prove this to be the case in Theorem~\ref{thm_stab}, and we use  analysis in \cite{SU-JFA09}, as well. 

We  present some heuristic arguments for our main assumption below. 
We will restrict $f$ below to a class of functions with support in some fixed compact $\mathcal{K}\subset \Omega$. Intuitively, to be able to recover all $f$  
supported in $\mathcal{K}$, we want for any $x\in \mathcal{K}$, at least one signal from $x$ to reach $\mathcal{G}$, i.e., we want to have a signal that reaches some $z\in\Gamma$ for $t<s(z)$. 
In other words, we should at least require that
\be{i2}
\forall x\in\mathcal{K}, \exists z\in\Gamma\; \mbox{so that} \; \dist(x,z)<s(z),
\ee
(one may want to have a non-strict equality above but we will not pursue this). In Theorem~\ref{thm_uq} below, we show that this necessary condition, up to replacing the $<$ sign by the $\le$ one,  is sufficient, as well. 

If we want that recovery to be stable, we need to be able to recover all singularities of $f$ ``in a stable way.''  
By the zero initial velocity condition, each singularity $(x,\xi)$ splits into two parts, see Proposition~\ref{pr_FIO} below: one that starts propagating in the direction $\xi$; and another one propagates in the direction $-\xi$. Moreover, neither one of those singularities  vanishes at $t=0$ (and therefore never vanishes), they actually start with equal amplitudes. For a stable recovery, we need to be able to detect at least one of them, in the spirit of  \cite{SU-JFA09}, i.e., at least one of them should reach $\mathcal{G}$.  This in particular allows us to reduce $T$ by half in the full boundary data case, i.e., when $\mathcal{G}=(0,T)\times\bo$, one can choose
\be{3.2}
T>T(\Omega)/2,
\ee
and still   hope that a stable recovery is possible. 
In the general case, define $\tau_\pm(x,\xi)$ by the condition
\[
\tau_\pm(x,\xi) = \max\left(\tau\ge0; \; \gamma_{x,\xi}(\pm\tau)\in \bar\Omega\right).
\]
Based on the arguments above, for stable recovery we should assume that $\mathcal{G}$ satisfies the following condition
\be{3.3}
\mbox{$\forall (x,\xi)\in S^*\mathcal{K}$,   $\left( \tau_\sigma(x,\xi), \gamma_{x,\xi}(\tau_\sigma(x,\xi)\right)\in\mathcal{G}$ for either $\sigma=+$ or $\sigma=-$ (or both).}
\ee
Compared to condition \r{i2}, this means that for each $x\in\mathcal{K}$ \textit{and each unit direction $\xi$}, at least one of the signals from $(x,\xi)$ and $(x,-\xi)$ reaches $\mathcal{G}$. This condition becomes necessary, if we replace $\mathcal{G}$ by its closure above, see Remark~\ref{remark_3}. 
In Theorem~\ref{thm_stab} below, we show that it is also sufficient. 

\subsection{Uniqueness.} We have the following uniqueness result, that in particular generalizes the result in \cite{finchPR} to the variable coefficients case.

\begin{remark}
Note that we do not need the geodesic flow to be non-trapping in this theorem since \r{i2} is a condition on a subset of the geodesics only. 
\end{remark}

\begin{theorem}  \label{thm_uniq}
Let $P=-\Delta$   outside $\Omega$, and let $\bo$ be strictly convex. 
Then under the assumption \r{i2}, if $\Lambda f=0$ on $\mathcal{G}$ 
for  $f\in H_D(\Omega)$ with $\supp f\subset\mathcal{K}$, then $f=0$. 
\end{theorem}

\begin{proof}
We follow the proof in \cite{finchPR}, where $g$ is Euclidean everywhere, and $T=\infty$ (actually, it is easy to see there that $T$ can be any number larger than $T(\Omega)$). We preserve the notation of \cite{finchPR} as much as possible. 

Recall that  $\dist(x,y)$ is the distance in the metric $c^{-2}g$. 
Let $d(x,y)$ be the (Euclidean) distance in $\R^n\setminus \Omega$ defined as the infimum of the Euclidean length of all smooth curves in $\R^n\setminus \Omega$ joining $x$ and $y$. The function $d$ is Lipschitz continuous, see \cite{finchPR}. Let $E_r(x)$ be the ball with center $x$ and radius $r>0$ in that metric. Then in \cite[Proposition~2]{finchPR}, Finch et al.\ proved the following  domain of dependence results for solutions vanishing on a part of $\bo$.

\begin{proposition}[\cite{finchPR}]  \label{prF}
Let $\Omega$ be an open bounded connected subset of $\R^n$ with a smooth boundary. Suppose $u$ is a smooth solution of the exterior problem
\begin{align*}
u_{tt}-\Delta u&=0, \quad t\in\R; \; x\in \R^n\setminus\Omega ,&\\
u&=h \quad  \mbox{on $\R\times \bo$}.
\end{align*}
Choose $p\not\in \Omega$, and $t_0<t_1$. If $u(t_0,\cdot)=u_t(t_0,\cdot)=0$  on $E_{t_1-t_0}(p)$, and $h$ is zero on 
\[
\left\{  (t,x); \; x\in\bo, \; t_0\le t\le t_1, \; d(x,p)\le t_1-t      \right\},
\]
then $u(t,p)=u_t(t,p)=0$ for all $t\in [t_0,t_1]$. 
\end{proposition}

Let $\Lambda f=0$ on $\mathcal{G}$, with $f$ as in the theorem, and let $u$ be the corresponding solution of \r{1}. Fix a point $x_0\in\mathcal{K}$. We will show that $f=0$ near $x_0$. By \r{i2}, there is  $p\in\bo$ so that $\dist(x_0,p)<s(p)$; then $(s(p),p)\in\mathcal{G}$. Let $0<\rho\ll1$ be such that $[0,s(p)-\rho]\times (E_\rho(p)\cap \bo) \subset \mathcal{G}$, and $\dist(x_0,q)<s(q)-\rho$, $\forall q\in E_\rho(p)\cap \bo$. We can therefore assume that 
\be{iA1}
\mathcal{G} = [0,T]\times \Gamma,  \mbox{where $\Gamma = E_\rho(p)\cap \bo$},
\ee
and 
\be{iA2}
\dist(x_0,q)<T \quad \forall q\in \Gamma.
\ee

The first step of the proof if to show that 
\be{i3}
f=0 \quad \mbox{in $B_\rho(p)$},
\ee
where $B_\rho(p)$ is the ball in the metric $g$ with center $p$ and radius $\rho$. The proof of \r{i3} is the same as in \cite{finchPR} with taking extra care about the range of the $t$ variable. Indeed, notice first that $u$ solves the wave equation in $\R^n\setminus\Omega$ with zero Cauchy data there, and boundary data $h=u|_{\R_+\times\bo}$ vanishing on $\mathcal{G}$, see \r{iA1}. Fix a small neighborhood $U$ of $p$ outside $\Omega$.  By \r{iA2} and the finite domain of dependence result in Proposition~\ref{prF}, we get  $u=0$ on $(-\rho+\eps,\rho-\eps)\times U$, where $0<\eps\to0$, when the size of $U$ tends to $0$. 

Next, $u$ solves the wave equation in the whole space, and can be extended (as a solution) as an even function of $t$. Therefore, by the unique continuation principle, see Theorem~\ref{thm_uq}, we get \r{i3}.

The next step is to iterate this argument and to prove  that $f=0$ near $x_0$. This would follow from the following property that we prove next: For some $\sigma>0$ independent of $\rho$, we have
\be{i4}
f=0 \quad \mbox{in $B_r(p)$}, \; r\ge\rho \quad  \Longrightarrow\quad f=0 \quad \mbox{in $B_{\min\{\rho+\sigma,T\}}(p)$}. 
\ee
The reason we did not just replace the minimum above with $\rho+\sigma$  is that we apply \r{i4} consecutively several times; at each step we gain $\sigma$, and we would like to make the radius equal to $T$. The last step needed for that might be smaller than $\sigma$ though, and \r{iA1}, \r{iA2} pose a restriction on how far we can go. 

Relation \r{i4} follows from the following. 

\begin{lemma} \label{lemma_claim}
Assume that $\supp f\subset K = \bar\Omega\setminus B_r(p)$ with some $r\ge \rho$. Let $\delta=\dist(E_\rho(p), K)$. 
Then $f=0$ in $B_{\min\{\rho+\delta,T\}}(p)$.
\end{lemma}

We prove Lemma~\ref{lemma_claim} below. Let $\alpha$ be the supremum of the distance $\dist(p,q)$, $q\in\Gamma$. Since $\bo$ is strictly convex, $\alpha<\rho$. Indeed, $\alpha$ is actually the maximum of those distances, if we replace $\Gamma$ by the compact $\bar\Gamma$. Then $\alpha=\dist(p,q_0)$ for some $q_0\in\bar\Gamma$. Because of the strict convexity  the latter is the length of the shortest geodesic on $\bo$ connecting $p$ and $q_0$. If we assume that $\alpha=\rho$, then that geodesic will be a minimizing curve for $c^{-2}g$ as well, therefore it will be a geodesic for that metric. That is impossible because for $\rho\ll1$, there is unique minimizing geodesic connecting $p$ and $q_0$, and that geodesic  cannot be on $\bo$. 

The following lemma  generalizes \cite[Propositon~5]{finchPR} to the current setting. We refer to  Fig.~1 that is similar to  Fig.~2.5 there for better understanding of the lemma and its proof.


\begin{figure}[t]   \label{fig:thermo}
  \centering
  \includegraphics[bb=0 0 626 309,width=4.82in,height=2.38in,keepaspectratio]{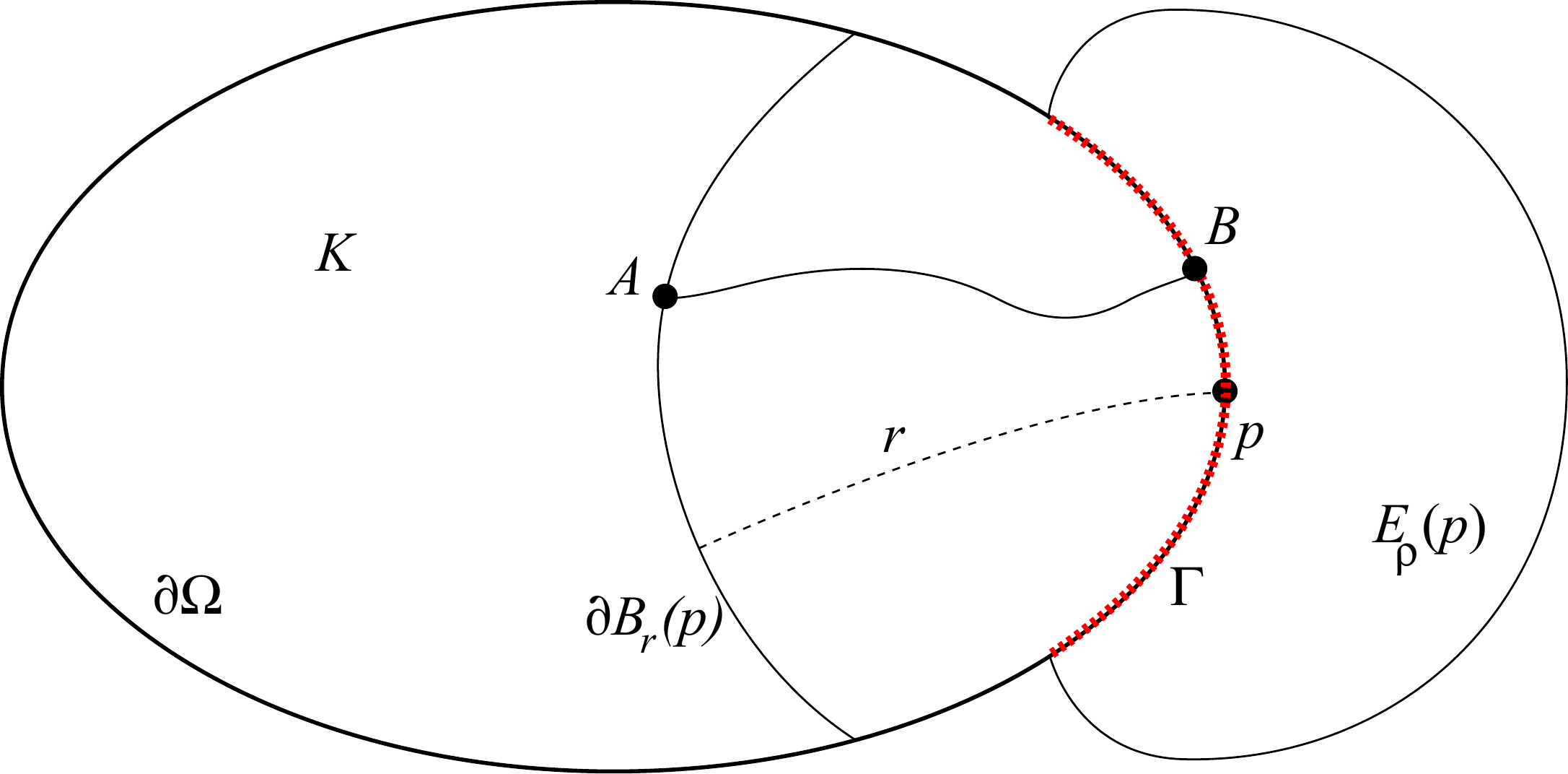}
  \caption{Illustrates Lemma~\ref{lemma_claim}. One can also show that $A\in \bo\cap \partial B_r(p)$. }
\end{figure}

\begin{lemma}\label{dist}
$\dist(K,\bar E_\rho(p))$ is the length of some geodesic segment joining a point 
$A\in K$ and a point $B\in\Gamma$. 
\end{lemma}

The proof is provided below, and we continue with the proof of Theorem~\ref{thm_uniq}. By Lemma~\ref{dist}, $\delta$ is the length of the geodesic segment connecting $A$ and $B$, as in the lemma. Then
\be{i4A}
\begin{split}
\rho+\delta &=\rho+\dist(A,B)= \dist(A,B)+\dist(B,p) +\left(\rho-\dist(B,p)\right)\\
   &\ge \dist(p,A) + \left(\rho-\dist(B,p)\right)\ge r+(\rho-\alpha). 
\end{split}
\ee
Note that $\sigma:= \rho-\alpha>0$ is independent of $r$. This proves the property \r{i4}, and therefore, the theorem.
\end{proof}

It remains to prove the two lemmas above.

\begin{proof}[Proof of Lemma~\ref{dist}]
We will provide a proof that is different and  shorter than that in \cite{finchPR}. Since $\dist(K,\bar E_\rho(p))$ is the distance between two compact sets, there is $A\in K$ and $B\in \bar E_\rho(p)$ so that $\dist(K,\bar E_\rho(p))=\dist(A,B)$. By the Hopf-Rynow theorem, there is a geodesic $\gamma$ connecting $A$, $B$ so that $\dist(A,B)$ is the length of $\gamma$. Clearly, $B$ belongs to $\partial E_\rho(p)$ that consists of two parts: the first one that we denote by $\partial E_\rho^{\rm{ext}}(p)$, that is outside $\bar\Omega$; and the second one is $\Gamma$, see \r{iA1} We will show first that $B$ must belong to the second one. Assume the opposite. Then $\gamma$ intersects $\bo$ once (because of the strict convexity) at some point $C\not\in\Gamma$ because if $C\in\Gamma$, then we would have $C=B$. Then the segment $CB$ of $\gamma$ is a straight line segment outside $E_\rho(p)$, see Figure~2.


\begin{figure}[h]   \label{fig2}
  \centering
  \includegraphics[bb=0 0 626 315,width=4.83in,height=2.43in,keepaspectratio]{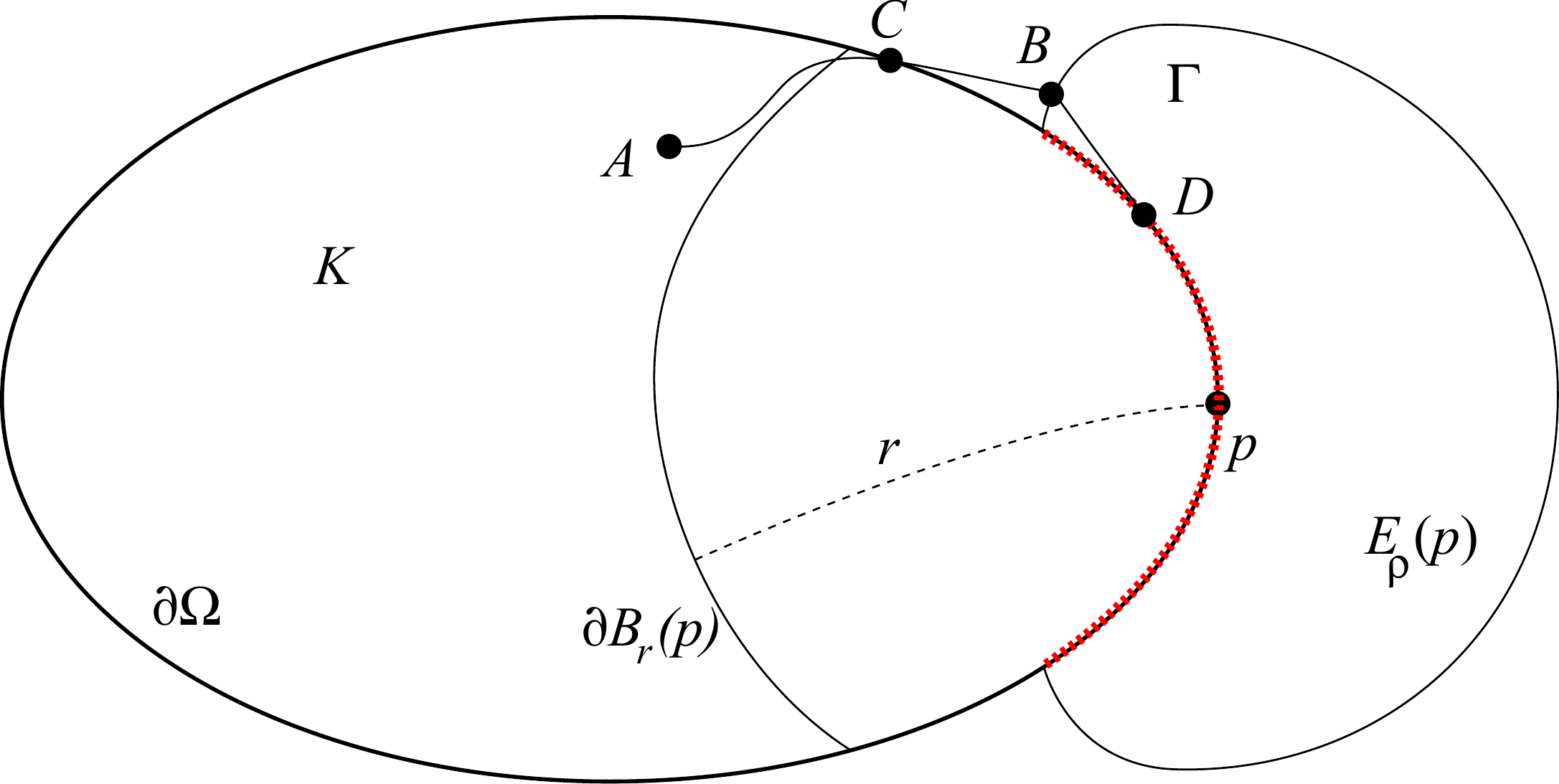}
  \caption{ }
\end{figure}

Let $c$ be the minimizing curve   for the metric $d$, lying outside $\Omega$,  that connects $B$ and $p$. It is easy to see (see \cite{finchPR}) that  $c$ exists and consists of a straight line segment $c_1=BD$ between $C$ and some $D\in \partial E_\rho(p)$, and a geodesic $c_2$ on $\bo$, possibly reduced to a point, so that $c_1$ and $c_2$ are tangent to each other and to $\bo$ at their common point  that we denote by $D$. Note that $\partial E_\rho^{\rm{ext}}(p)$ is an open surface, therefore $B\not=D$. 
Then the curve $CB\cup BD$ locally minimizes the lengths of all curves connecting $C$ and $D$ with the property that they consist of a curve outside $E_\rho(p)\cup \Omega$ connecting $C$ and some $B'\in \partial E_\rho^{\rm{ext}}(p)$ close to $B$; and then another curve,  outside $\Omega$ but inside  $E_\rho(p)$, connecting $B'$ to $D$. 
Then $CB\cup BD$ must be a straight line segment; otherwise we can make it shorter by an arbitrary small perturbation, and that would contradict the minimizing property above. That segment is tangent to $\bo$. 
 By the strict convexity of $\Omega$, it cannot have two common points $C$ and $D$ with $\bo$.  This contradiction shows that  $B\in \Gamma$, and this proves the second statement of the lemma. 
\end{proof}

\begin{proof}[Proof of Lemma~\ref{lemma_claim}]
Roughly speaking, the idea of the proof is that we can apply the arguments at the beginning of the proof of the theorem by shifting the initial moment form $t=0$ to $t=\delta$. 

First, by the definition of $\delta$ and the standard domain of dependence argument, 
\be{i5}
u=u_t=0 \quad \mbox{on $[-\delta,\delta]\times E_\rho(p)$}.
\ee
Let $U$ be a small enough neighborhood of $p$ in $E_\rho(p)$. If $\delta+\rho\le T$, by the domain of dependence argument for the exterior problem \cite[Proposition~2]{finchPR}, in view of \r{iA1}, \r{iA2}, $u=0$ on $[\delta,\delta+\rho-o(1)]\times U$, where by $o(1)$ we  denote terms tending to $0$ when the size of $U$ tends to $0$. If $\delta+\rho>T$, then we can prove that only in the time interval $[\delta,T-o(1)]$. Therefore, in both cases,  the time interval is $[\delta,\min\{\delta+\rho,T\}-o(1)]$.  Since $u$ extends as an even solution in the $t$ variable, we get that $u=0$ for $|t|\le \min\{\delta+\rho,T\}-o(1)$, $x\in U$. Then from the unique continuation result in Theorem~\ref{thm_uq}, $u|_{t=0}$, and therefore $f$ vanishes in $B_{\min\{\delta+\rho,T\}-o(1)}(p)$. Letting $U$ tend to $p$, we get that $f=0$  in $B_{\min\{\delta+\rho,T\}}(p)$. 
\end{proof}

It is probably worth mentioning that we actually proved the following result about partial recovery given insufficient information.

\begin{proposition}
Let $P=-\Delta$ outside $\Omega$, and let $\bo$ be strictly convex. Assume that $\Lambda f=0$ on $\mathcal{G}$ for some $f\in H_D(\Omega)$ with $\supp f\subset\Omega$ with  $\mathcal{G}$ as in \r{i0} that may not satisfy \r{i2}. Then $f=0$ in $W$, where
\[
W := \left\{ x\in   \Omega; \; \mbox{\rm $\exists  z\in\Gamma$ so that $\dist(x,z)<s(z)     $} \right\}.
\]
Moreover, no information about $f$ on $\Omega\setminus \bar W$ is contained in $\Lambda f|_{\mathcal{G}}$. 
\end{proposition}

\subsection{Stability.} In this section, we use tools from microlocal analysis. We refer, for example,  to \cite{Treves} for an introduction to the theory of pseudo-differential operators (\PDO s) and to \cite{Treves2, Duistermaat} for the theory of Fourier Integral Operators (FIOs).

We now consider the situation where $\Lambda f$ is given on a set $\mathcal{G}$ satisfying \r{3.3}. Since $\mathcal{K}$ is compact and $\mathcal{G}$ is closed, one can always choose $\mathcal{G'}\Subset\mathcal{G}$ that still satisfies \r{3.3}. Fix $\chi\in C_0^\infty([0,T]\times \bo)$ so that $\supp\chi\subset\mathcal{G}$ and $\chi=1$ on $\mathcal{G}'$. The measurements are then modeled by $\chi \Lambda f$, which depends on  $\Lambda f$ on $\mathcal{G}$ only.

We start with a description of the operator $\Lambda$ that is of independent interest as well. In the next proposition, we formally choose $T=\infty$. 

\begin{proposition}\label{pr_FIO}
$
\Lambda = \Lambda_+ +\Lambda_-, 
$
where 
$\Lambda_{\pm} : C_0^\infty(\Omega) \to C^\infty((0,\infty)\times \bo)$ are elliptic Fourier Integral Operators of zeroth order with canonical relations given by the graphs of the maps 
\be{C}
(y,\xi)\mapsto \left(\tau_\pm(y,\xi), \gamma_{y,\pm\xi}(\tau_\pm(y,\xi)), |\xi|, \dot \gamma'_{y,\pm\xi}(\tau_\pm(y,\xi))\right),
\ee
where $|\xi|$ is the norm in the metric $c^{-2}g$, and the prime in $\dot\gamma'$ stands for the tangential projection of $\dot\gamma$ on $T\bo$.
\end{proposition}

\begin{proof}
This statement is well known and follows directly from \cite{Duistermaat}, for example. We will give more details that are needed just for the proof of this proposition in order to be able to compute the principal symbol in Theorem~\ref{thm_stab}. 

We start with a standard geometric optics construction. Fix $x_0\in \Omega$. In a neighborhood of $(0,x_0)$, the solution to \r{2} is given by
\be{o1}
u(t,x) =  (2\pi)^{-n} \sum_{\sigma=\pm}\int e^{\i\phi_\sigma(t,x,\xi)} a_\sigma(x,\xi,t) \hat f(\xi)\, \d \xi,
\ee
modulo smooth terms, where the phase functions $\phi_\pm$ are positively homogeneous of order $1$ in $\xi$ and solve the eikonal equations
\be{o2}
\mp\partial_t\phi_\pm = |\d_x\phi_\pm|,\quad \phi_\pm|_{t=0}=x\cdot\xi,
\ee
while $a_\pm$ are classical amplitudes of order $0$ 
solving  the corresponding transport equations, see \cite[p.~128]{Duistermaat} or \cite[eqn.~(VI.1.50)]{Treves2}. 
In particular, $a_\pm$ satisfy
\[
a_+ +a_-=1\quad \mbox{for $t=0$}.
\]
Since $\partial_t\phi_\pm=\mp\xi$ for $t=0$, and $u_t=0$ for $t=0$, we also see that
\[
a_+=a_- \quad \mbox{for $t=0$}.
\]
Therefore, $a_+=a_-=1/2$ at $t=0$. Note that if $P=\Delta$, then $\phi_\pm = x\cdot\xi \mp t|\xi| $, and $a_+\equiv a_-=1/2$. The principal term $a_\pm^{(0)}$ of $a_\pm \sim\sum_{j\ge0} a_\pm^{(-j)}$ satisfies the homogeneous transport equation
\be{tr}
\left( \partial_t - c^{2}g^{ij}(\partial_{x^j}\phi_\pm) \partial_{x^j}+C_\pm \right)a_\pm=0,\quad a_\pm|_{t=0}=1/2,
\ee
where $C_j$ depend on the coefficients of $P$ and on $\phi_\pm$, see \cite[eqn.~(VI.1.49)]{Treves2}.

By the stationary phase method, singularities starting from $(x,\xi)\in \WF(f)$ propagate along geodesics in the phase space  issued from $(x,\xi)$, for $\sigma=+$. i.e., they stay on the curve $(\gamma_{x,\xi}(t),\dot\gamma_{x,\xi}(t) )$; and from $(x,-\xi)$, for $\sigma=-$, i.e., they stay on the curve $(\gamma_{x,-\xi}(t), \dot\gamma_{x,-\xi}(t))$. This is consistent with the general propagation of singularities theory for the wave equation because the principal symbol of the wave operator $\tau^2-c^{2}|\xi|_g$ has two roots $\tau = \pm c|\xi|_g$.

The construction is valid as long as the eikonal equations are solvable, i.e., along geodesics issued from $(x,\pm\xi)$ that do not have conjugate points. Assume that $\WF(f)$ is supported in a small neighborhood of $(x_0,\xi_0)$ with some $\xi_0\not=0$. 
 Assume first that the geodesic from $(x_0,\xi_0)$ with endpoint on $\bo$ has no conjugate points.  We will study the $\sigma=+$ term in \r{o1} first. Let $\phi_{\rm b}$, $a_{\rm b}$ be the restrictions of $\phi_+$, $a_+$, respectively, on $\R\times\bo$. Then, modulo smooth terms, 
\be{FIO1}
\Lambda_+ f:= u_+(t,x)|_{\R\times\bo} = (2\pi)^{-n}  \int e^{\i\phi_{\rm b}(t,x,\xi)} a_{\rm b}(x,\xi,t) \hat f(\xi)\, \d \xi,
\ee
where $u_+$ is the $\sigma=+$ term in \r{o1}. 
Set $t_0 = \tau_+(x_0,\xi_0)$, $y_0=\gamma_{x_0,\xi_0}(t_0)$, $\eta_0= \dot \gamma_{x_0,\xi_0}(t_0)$; in other words, $(y_0,\eta_0)$ is the exit point and direction of the geodesic issued from $(x_0,\xi_0)$ when it reaches $\bo$. Let $x=(x',x^n)$ be boundary normal coordinates near $y_0$. 
Writing $\hat f$ in \r{FIO1} as an integral, we see that \r{FIO1} is an oscillating integral with phase function $\Phi= \phi_+(t,x',0,\xi) -y\cdot\xi$. Then (see \cite{Treves2}, for example), the set $\Sigma:= \{\Phi_\xi=0\}$ is given by the equation
\[
y=\partial_\xi \phi_+(t,x',0,\xi) 
\]
It is well known, see e.g., Example~2.1 in \cite[VI.2]{Treves2}, that this equation implies that $(x',0)$ is the endpoint of the geodesic issued from $(y,\xi)$ until it reaches the boundary, and $t=\tau_+(y,\xi)$, i.e., $t$ is the time it takes to reach $\bo$. In particular, $\Sigma$ 
is a manifold of dimension $2n$, parametrized by $(y,\xi)$. Next,  the map
\be{C2}
\Sigma \ni (y,t,x',\xi)  \longmapsto \left(y,t,x',-\xi, \partial_t\phi_+, \partial_{x'}\phi_+\right) 
\ee
is smooth  of rank $2n$ at any point. This shows that $\Phi$ is a non-degenerate phase, see \cite[VIII.1]{Treves2}, and that $f\mapsto \Lambda_+f$ is an FIO associated with the Lagrangian given by the r.h.s.\ of \r{C2}. The canonical relation is then given by
\[
C := \left(  y,\xi,    t,x',   \partial_t\phi_+, \partial_{x'}\phi_+     \right)
 ,\quad  (y,t,x',\xi)\in  \Sigma .
\]
Then \r{C} follows from the way $\phi_+$ is constructed by the Hamilton-Jacobi theory. The proof in the $\sigma=-$ case is the same. 

The proof above was done under the assumption that there are no conjugate points on $\gamma_{y_0,\xi_0}(t)$, $0\le t\le \tau_+(y_0,\xi_0)$. To prove the theorem in the general case, let $t_1\in (0,\tau_+(y_0,\xi_0))$ be such that there are no conjugate points on that geodesic for $t_1\le t\le \tau_+(y_0,\xi_0)$. Then each of the terms in \r{o1} extends to a global elliptic FIO mapping initial data at $t=0$ to a solution at $t=t_1$, see e.g., \cite{Duistermaat}. Its canonical relation is the graph of the geodesic flow between those two moments of time (for $\sigma=+$, and with obvious sign changes when $\sigma=-$). 
We can compose this with the local FIO constructed above, and the result is a well defined elliptic  FIO of order $0$ with canonical relation \r{C}. 
\end{proof}

Choose and fix $T>\sup_\Gamma s$, see \r{i1}.  Let $A$ be the ``back-projection'' operator defined in \r{2} and \r{4}. Note that $A$ is always applied to $\chi\Lambda$ below, therefore $\phi=0$ in this case.

\begin{theorem} \label{thm_stab} 
$A\chi\Lambda$ is a zero  order classical \PDO\ in some neighborhood of $\mathcal{K}$ with principal symbol 
\[
\frac12\left( \chi(\gamma_{x,\xi}(\tau_+(x,\xi))) +  \chi(\gamma_{x,\xi}(\tau_-(x,\xi)) )\right).
\] 
If $\mathcal{G}$ satisfies \r{3.3}, then 

(a) $A\chi\Lambda$ is elliptic, 

(b) $A\chi\Lambda$ is a Fredholm operator on $H_D(\mathcal{K})$, and 

(c) there exists a constant $C>0$ so that
\be{F}
\|f\|_{H_D(\mathcal{K})}\le C \|\Lambda f\|_{H^{1}(\mathcal{G})}.
\ee
\end{theorem}

\begin{remark}\label{remark_3}
By \cite[Proposition~3]{SU-JFA09}, condition \r{3.3}, with $\mathcal{G}$ replaced by its closure, is  a necessary condition for stability in any pair of Sobolev spaces. In particular, $c^{-2}g$ has to be non-trapping for stability. Indeed, then the proof below shows that $A\chi\Lambda$ will be a smoothing operator on some non-empty open conic  subset of $T^*\mathcal{K}\setminus 0$.  
\end{remark}

\begin{remark}
Note that $\Lambda : H_D(\mathcal{K}) \to H^1([0,T]\times \bo)$ is bounded. This follows for example from Proposition~\ref{pr_FIO}.
\end{remark}
\begin{proof}
We will use the geometric optics construction in the proof of Proposition~\ref{pr_FIO}, using the notation there.

To construct a parametrix for $A\chi \Lambda f$, we apply a geometric optic construction again, using the two characteristic roots $\pm c|\xi|_g$. The boundary data $\Lambda_+f$ have a wave front set in a small conic neighborhood of  $((t_0,y_0'),(1, \eta_0'))$. Note that $\eta_0^n\not=0$ because geodesics issued from $\mathcal{K}$ cannot be tangent to $\bo$. Then for the solution $v$ of \r{2} with $h=\Lambda_+f$, 
we can apply the geometric optics construction above, but now with initial condition on $\R\times\bo$, to get  two types of singularities starting from that one. The first one propagates along the geodesics close to $\gamma_{x_0,\xi_0}$ in the opposite direction. The second one propagates along the geodesic close to the one issued from $((t_0,y_0), (\eta',-\eta^n))$, that is transversal to $\bo$. This ray is in fact a reflected  $\gamma_{x_0,\xi_0}$.  
By the propagation of singularities results, those singularities  stay on those geodesics until they reach $\bo$ again, then reflect by law of geometric optics, etc., i.e., they propagate along the broken geodesics issued from a neighborhood of that point. Near $t=T$ however, the solution to \r{2}, where $\phi=0$, is zero because we have zero Cauchy data, and $h=\chi \Lambda f=0$ for $t$ close to $T$. This shows that the second types of singularities do not exist; and  in our parametrix construction, we need to work with the first one only. 

We look for a parametrix of the solution of the wave equation \r{2} with zero Cauchy data at $t=T$ and boundary data $\chi\Lambda_+ f$ in the form
\[
v(t,x) = (2\pi)^{-n}  \int e^{\i\tilde \phi(t,x,\xi)} b(x,\xi,t) \hat f(\xi)\, \d \xi.
\]
The arguments above show that $\tilde\phi=\phi_+$. Next, for $x\in\bo$, we have $b=\chi a$. We need to find $b$ at $t=0$. 
The amplitude $b$ satisfies the same transport equation as in the proof of Proposition~\ref{pr_FIO} but with initial condition at $\R\times\bo$. In particular, it is a classical amplitude of order $0$. Let $b_0$ be its principal part. Then $b_0$ satisfies \r{tr}, also  satisfied by $a^{(0)}_+$,  that is a linear homogeneous ODE along the bicharacteristic close to $(\gamma_{x_0,\xi_0}, \dot \gamma_{x_0,\xi_0})$. Therefore, $b_0$ is a linear function of its initial condition at $\R\times\bo$. 
If we assume for a moment that $\chi=1$, then we would get $b_0=a^{(0)}_+$, therefore, $b_0=1/2$ for $t=0$. 
Therefore,  we get that $b_0(x,\xi)|_{t=0}$ is given by the value of $\chi/2$ at the exit point of $\gamma_{x,\xi}$ on $\bo$ because that value is the initial condition of the transport equation on that bicharacteristic. 

The arguments above reveal the geometry of the singularities but some of them are not needed for the formal proof. One can define $v$ as above, localized near the bicharacteristic issued from $(x_0,\xi_0)$, 
and let $u_+$ be the solution of \r{2} with $\phi=0$ and $h=\chi \Lambda_+f$. Then one easily checks that $w := u_+-v$ solves the wave equation modulo smooth terms, with smooth boundary condition, and that  $w=0$ near  $t=T$; and is therefore smooth. 

In the same way one treats the $\sigma=-$ term. This proves the theorem assuming no conjugate points in $\Omega$.

In the general case, we can apply those arguments step by step, in intervals $[0,t_1]$, then $[t_1,t_2]$, etc., short enough so that there are no conjugate points on the corresponding geodesic segments. After the first step, we get $(u,u_t)$ at $t=t_1$. Then we construct a parametrix from $t=t_1$ to $t=t_2$ using a new phase function.  Note that now, when $\sigma=+$, for example,  $u_t|_{t=t_1}$ does not vanish anymore. On the other hand, $(u,u_t)|_{t=t_1}$ is Cauchy data of a solution which singularities do not travel in two opposite directions, and we will still get one term only, that is an analogue of the $\sigma=+$ one in \r{o1}. Then we reach the boundary and apply the result above. Next, step by step, we go back to the hyperplane $t=0$. 
An alternative way is to apply the Egorov's theorem from $t=0$ to $t=\tilde t$, instead of the partition of the time interval, where $\tilde t$ is such that there are no conjugate points on the bicharacteristic issued from $(x_0,\xi_0)$ from $\tilde t$ to $\tau_+(x_0,\xi_0)$; and on that segment, we use the arguments above.  

This proves the first statement of the theorem. 

Parts (a), (b) follows immediately from the ellipticity of $A\chi\Lambda$ that is guaranteed by \r{3.3}.

To prove part (c), note first that the ellipticity of $A\chi\Lambda$  and the mapping property of $A$, see  \cite{LasieckaLT}, imply the estimate
\[
\|f\|_{H_D(\mathcal{K})}\le C\left(\|\chi\Lambda f\|_{H^1}+  \|f\|_{L^2 }     \right).
\]
By Theorem~\ref{thm_uniq}, and \r{3.3}, $\chi\Lambda$ is injective on $H_D(\mathcal{K})$. By \cite[Proposition~V.3.1]{Taylor-book0}, one gets 
estimate \r{F} with a  constant $C>0$ possibly different than the one above. 
\end{proof}

\begin{corollary}\label{cor_1}  
Let $g$ be Euclidean outside $\Omega$, and let $\bo$ be strictly convex. Then if $\Lambda f=0$ on $[0,T]\times\bo$ for some  $f\in H_D(\Omega)$, with $T>T(\Omega)/2$, then $f=0$.  
\end{corollary}

\section{Thermoacoustic tomography and integral geometry} \label{sec_ig} 
If $P=-\Delta$, and if $n$ is odd, the solution of the wave equation can be expressed in terms of spherical means, as it is well known. Then the problem can be formulated as an integral geometry problem --- recover $f$ from integrals over spheres centered at $\bo$, with radii in $[0,T]$, and this point of view has been exploited a lot in the literature. One may attempt to apply the same approach in the variable coefficients case; then one has to integrate over geodesic spheres. This has two drawbacks. First, those integrals represent the leading order terms of the solution operator only, not the whole solution. That would still be enough for constructing a parametrix however but not the Neumann series solution in Theorem~\ref{thm2.1}.  The second problem is that the geodesic spheres become degenerate in presence of caustics. The wave equation viewpoint that we use in this paper is not sensitive to caustics. We still have to require that the metric be non-trapping in some of our theorems. By the remark following Theorem~\ref{thm_stab} however, this is a necessary condition for stability. On the other hand, it is not needed for the uniqueness result as long as \r{i2} is satisfied. 

\section{Acknowledgments.} The authors thank Peter Kuchment for  his comments on a preliminary version of this work. We also thank the referees whose suggestions helped improved the exposition.

\appendix
\section{Unique continuation for the wave equation}
We recall here a Holmgren's type of theorem for the wave equation $(\partial_t^2+P)u=0$ due mainly to Tataru.  While this theorem is well known and used, and follows directly from the results cited below, we did not find it clearly formulated in the literature. 

\begin{theorem}\label{thm_uq}
Let $P$ be the differential operator in $\R^n$ as in the Introduction. Assume that $u\in H^1_{\rm loc}$ satisfies 
\[
(\partial_t^2+P)u=0
\]
in a neighborhood of $[-T,T]\times \{x_0\}$, with some $T>0$, $x_0\in \R^n$.  Then 
\[
u(t,x)=0 \quad \mbox{for} \quad |t|+\dist(x_0,x) \le T.
\]
\end{theorem}

\begin{proof}
If $P$ has analytic coefficients, this is Holmgren's theorem. In the non-analytic coefficients case, a version of this theorem was proved by Robbiano \cite{Robbiano91} with $\rho$ replaced by $K\rho$ with an unspecified constant $K>0$. It is derived there from a local unique continuation theorem across a surface that is ``not too close to being characteristic". In \cite{Hormander92}, H\"ormander showed that one can choose $K= \sqrt{27/23}$, in both the local theorem \cite[Thm~1]{Hormander92}   and the  global theorem \cite[Corollary~7]{Hormander92}. Moreover, he showed that $K$ in the global one can be chosen to be the same as the $K$ in the local one. Finally, Tataru \cite{tataru95, tataru99} proved a unique  continuation result that implies unique continuation across any non-characteristic surface. This shows that actually  $K=1$ in H\"ormander's work, and the theorem above then follows from \cite[Corollary~7]{Hormander92}.
\end{proof}

%


\end{document}